\documentclass{amsart}
\usepackage{graphicx}
\usepackage{amssymb}
\usepackage{array}
\usepackage{enumerate}
\usepackage{url}
\usepackage{algorithm}

\newtheorem{theorem}{Theorem}

\newtheorem{lemma}[theorem]{Lemma}

\numberwithin{theorem}{section}
\theoremstyle{remark}
\newtheorem*{remark}{Remarks}

\begin{document}

\title[An explicit van der Corput estimate]{An explicit van der Corput 
estimate for $\zeta(1/2+it)$}
\author[G.A. Hiary]{Ghaith A. Hiary}
\thanks{Preparation of this material is partially supported by
the National Science Foundation under agreements No. 
 DMS-1406190.}
\address{Department of Mathematics, The Ohio State University, 231 West 18th
Ave, Columbus, OH 43210.}
\email{hiaryg@gmail.com}
\subjclass[2010]{Primary 11Y05.}
\keywords{Van der Corput estimate, exponential sums, Riemann zeta function.}

\begin{abstract}
An explicit estimate for the Riemann zeta function on the critical line is derived
using the van der Corput method. An explicit van der Corput lemma is presented.
\end{abstract}

\maketitle

\section{Introduction} \label{intro}

The Riemann zeta function is defined for $s=\sigma+it$ by
\begin{equation}\label{dirichlet def}
\zeta(s) := \sum_{n=1}^{\infty} n^{-s}, \qquad \sigma>1.
\end{equation}
It can be analytically continued everywhere
except for a simple pole at $s=1$. 
A well-known problem in number theory is to bound the growth rate
of zeta on the critical line $\sigma = 1/2$. 
This problem has led to deep 
ideas in the theory of exponential sums. In
particular, the method of exponent pairs (see \cite{graham-kolesnik}, 
\cite[page 116]{titchmarsh}), and the Bombieri-Iwaniec
method~\cite{bombieri-iwaniec-1,bombieri-iwaniec-2}. 

Since $|\zeta(\sigma+it)| = |\zeta(\sigma-it)|$, 
 we can suppose that $t\ge 0$.
Starting with the series \eqref{dirichlet def},
it follows by the Euler-Maclaurin formula that $\zeta(1/2+it) \ll t^{1/2}$ 
for large enough $t$.
This can be improved substantially 
by appealing to the Riemann--Siegel formula, which gives 
 $\zeta(1/2+it)\ll t^{1/4}$.
The Weyl-Hardy-Littlewood method (see
\cite[section 5.3]{titchmarsh}) detects
a certain amount of cancellation in
the main sum of the Riemann--Siegel formula,
further improving the bound to $\zeta(1/2+it)\ll
t^{1/6}\log^{3/2} t$.
The van der Corput method (e.g.\ \cite{corput-koksma,corput-1929})
removes an extra $\sqrt{\log t}$ factor,
which sharpens the estimate to 
\begin{equation}\label{corput estimate 1}
\zeta(1/2+it)\ll t^{1/6}\log t.
\end{equation}

The exponent $1/6$ in estimate \eqref{corput estimate 1} is hard to improve. 
The sharpest result so far is due to Huxley~\cite{huxley}, 
who proved that $\zeta(1/2+it)\ll t^{32/205}\log^{\gamma} t$ for some constant
$\gamma$. Note $32/205 = 0.15609\ldots$. (See also the recent result in \cite{bourgain}.)
Assuming the Riemann hypothesis, one can show
that $\zeta(1/2+it)\ll \exp (A\log t/\log \log t)$
for some constant $A$; see \cite[\textsection{14.4}]{titchmarsh}.
Hence,
$\zeta(1/2+it)$ is conjectured to grow slower than any fixed power of $t$.

In this article, we 
obtain an explicit bound of the van der Corput type \eqref{corput estimate 1}.
That is, we compute constants $C_1$ and $C_2$ such that 
$|\zeta(1/2+it)| \le C_1 t^{1/6}\log t$ for $t \ge C_2$.
This is part of our work in progress about subconvexity 
bounds for zeta. We are 
 also motivated by the following computational considerations where 
the simplicity of an explicit van der Corput estimate is
particularly attractive.

Specifically, in numerical tests of the growth rate of $\zeta(1/2+it)$, 
 it is necessary to have explicit bounds. For one usually cannot distinguish 
a small power of $t$ from a logarithm or a subexponential factor
unless $t$ is prohibitively large. An explicit bound offers
an unconditional measure
against which one can compare 
large values of $|\zeta(1/2+it)|$ 
found by special numerical searches, such as those in \cite{zeta-comp}.

Another motivation comes from 
the algorithms in \cite{hiary}
which employ a Taylor expansion
to express zeta as a sum
of low degree exponential sums 
and their derivatives.
The size of the remainder term in these algorithms 
is bounded by the highest order derivative used  and  
the maximal size of a certain subdivision of the main 
sum of zeta. 
In this context, an improved and simple method to obtain 
an explicit bound is of value as it enables reducing 
the number of derivatives needed to guarantee a given error tolerance. 
This in turn will improve the running time appreciably.

Explicit bounds for zeta were obtained by 
Cheng and Graham \cite{cheng-graham}, and 
recently by Platt and Trudgian \cite{platt-trudgian} who proved
that 
\begin{equation}\label{pt est}
|\zeta(1/2+it)|\le 0.732 t^{1/6}\log t, \qquad (t\ge 2).
\end{equation}
We improve the leading constant in \eqref{pt est} by 14\%.
\begin{theorem}\label{zeta bound}
If $0\le t \le 3$, then $|\zeta(1/2+it)| \le 1.461$. 
If $t \ge 3$, then  
\begin{equation}\label{zeta bound eq}
|\zeta(1/2+it)| \le 0.63\, t^{1/6}\log t.
\end{equation}
\end{theorem}
A numerical computation reveals that 
$\max_{t\ge 3} |\zeta(1/2+it)|/t^{1/6}\log t > 0.507$.
So, a priori, the leading constant in \eqref{zeta bound eq} cannot 
be reduced below $0.507$ without an additional assumption on the size of $t$.
In fact, if one employs the Riemann--Siegel--Lehman bound
$|\zeta(1/2+it)|\le 4(t/(2\pi)^{1/4}-2.08$
from Lemma~\ref{trivial bound lemma}
for any range of $t$ at all, then the leading constant
cannot break $0.541$, as a numerical computation shows that
\begin{equation}\label{riemann siegel barrier}
\min_{t\ge 3} \frac{4(t/(2\pi)^{1/4}-2.08}{t^{1/6}\log t} > 0.541.
\end{equation}

Our proof of Theorem~\ref{zeta bound} uses a 
subdivision of the Riemann--Siegel 
main sum different than \cite{cheng-graham,platt-trudgian},
giving rise to a simpler optimization problem.
In particular, we divide the main sum into short
pieces of length $\approx t^{1/3}$ (which appears to be a new, and natural, subdivision;
c.f.\ Weyl's method in \cite[Section 5.3]{titchmarsh} which divides the main sum 
into pieces of length $\lesssim t^{1/6}$). 
This subdivision enables better control of the oscillations 
in each piece
via Lemma~\ref{corput lemma} since the range of 
$f'''(x)$  for us will be more restricted. 
\begin{lemma}\label{corput lemma}
Let $f(x)$ be a real-valued function with three continuous derivatives on
$[N+1,N+L]$. Suppose there are $W > 1$ and $\lambda \ge 1$ such that
$\frac{1}{W} \le |f'''(x)| \le \frac{\lambda}{W}$
for $N+1\le x \le N+L$. If $\eta>0$, then 
\begin{equation}
\Big|\sum_{n=N+1}^{N+L} e^{2\pi i f(n)} \Big|^2 \le 
(L W^{-1/3} +\eta)
(\alpha L + \beta W^{2/3}), 
\end{equation}
where 
\begin{equation}\label{alpha beta formula}
\begin{split}
\alpha  &:=\alpha(W,\lambda,\eta)=
 \frac{1}{\eta} +\frac{64\lambda}{75} \sqrt{\eta+W^{-1/3}}+
\frac{\lambda \eta}{W^{1/3}}+  \frac{\lambda}{W^{2/3}},\\
\beta &:=\beta(W,\eta)= \frac{64}{15}\frac{1}{\sqrt{\eta}} + \frac{3}{W^{1/3}}.
\end{split}
\end{equation}
\end{lemma}

\begin{remark}
Lemma~\ref{corput lemma} is proved in \textsection{\ref{lemma proof}}. 
This lemma is an explicit version of the process
$AB$ in the method of exponent pairs. 
The condition
$W>1$ in the lemma can be relaxed to $W> 2/\pi$. 
The number 
$\eta$ will be chosen of size $\approx 1$
in our application, Theorem~\ref{zeta bound}. 
\end{remark}

The proof of Theorem~\ref{zeta bound} completely overlooks cancellation 
among the $\approx t^{1/6}$ pieces where Lemma~\ref{corput lemma} is applied. 
This is the main source of inefficiency in our proof. 
As far as we know, there is
no definitive mechanism to take advantage of this cancellation.
 Instead,
one works with longer pieces, 
and tries to prove
cancellation within each one.

Of course, the bound \eqref{zeta bound eq} is asymptotically 
 far from the truth.
And even for moderately large values of $t$, this bound  
is still probably a
substantial overestimate. 
Evidence for this comes from computations by J.~W.\ Bober 
and the author, some of which are summarized in \cite{zeta-comp}.
In these computations, several hundred large values of zeta
were recorded by 
computing $\zeta(1/2+it)$ at certain special points.
The largest value found this way was 
\begin{equation}\label{largest value}
|\zeta(1/2+iT_0)| = 16244.86526\ldots,
\end{equation}
where $T_0=39246764589894309155251169284104.050622\ldots$,
so $T_0\approx 3.9\times 10^{31}$.
(To our knowledge, \eqref{largest value} is the largest value of zeta 
computed so far.)
In comparison, the bound \eqref{zeta bound eq} gives
$|\zeta(1/2+iT_0)| \le 8448744$, which is about $521$ times larger
than \eqref{largest value}.
Alternatively, one can use the 
van der Corput bound \eqref{vdc bound} directly.
This gives a bound that is about $507$ times larger 
than our computed value.
It should be stressed, though, that \eqref{largest value}
was found by searching  a thin set of special points,
not by an exhaustive search, and so
only provides a lower bound for $\max_{t\in [0,T_0]} |\zeta(1/2+it)|$.

\section{Proof of Theorem~\ref{zeta bound}}

For the remainder of the paper, we set
\begin{equation}
c_0 :=0.63,\qquad t_0 := 9.3\times 10^7.
\end{equation}

\begin{proof}
We consider $t$ over three ranges.
In the range $0\le t \le 200$, we use the computational bound 
from Lemma~\ref{computational bound lemma}.
In the intermediate range 
$200\le t \le t_0$ we use the Riemann--Siegel--Lehman bound supplied by 
Lemma~\ref{trivial bound lemma}. 
And in the last range $t\ge t_0$ we use the van der Corput bound
supplied by  Lemma~\ref{corput bound lemma}. 
\end{proof}

To handle the intermediate and last ranges of $t$, we rely on 
 Lemma~\ref{riemann siegel lemma},
which is consequence of the Riemann--Siegel formula.
This lemma requires $t\ge 200$, which is the reason we treat
the initial range $t\in[0, 200]$ separately.  
\begin{lemma}\label{riemann siegel lemma}
If $t \ge 200$ and $n_1=\lfloor \sqrt{t/(2\pi)}\rfloor$, then 
\begin{equation}\label{rs lemma}
|\zeta(1/2+it)| \le 2|\sum_{n=1}^{n_1} n^{-1/2+it}| 
+\mathcal{R}(t).
\end{equation}
where $\mathcal{R}(t):= 1.48 t^{-1/4}  + 0.127 t^{-3/4}$.
\end{lemma}
\begin{proof}
The lemma gives a small improvement on \cite[Lemma 3]{platt-trudgian}.
The improvement is that the inequality for $\mathcal{R}(t)$ is
tighter if $t\ge 200$. The lemma is stated separately for emphasis, 
as it is an essential first step in all that follows.

We apply the triangle inequality to 
the Riemann--Siegel formula in \cite[page 9]{gabcke-thesis} to obtain
\begin{equation} \label{eq:rsform}
|\zeta(1/2+it)| \le 2 \left|\sum_{n=1}^{n_1} 
\frac{\cos(t\log n-\theta(t))}{\sqrt{n}}\right| 
+ \beta \left(\frac{2\pi}{t}\right)^{1/4} 
+ |R_0(t)|,
\end{equation}
where $n_1=\lfloor \sqrt{t/2\pi}\rfloor$, $|R_0(t)| < 0.127 t^{-3/4}$ for $t \ge 200$, and
\begin{equation}
\beta = \max_{0\le z\le 1} 
\left|\frac{\cos \frac{\pi}{2}(z^2+3/4)}{\cos \pi z}\right|.
\end{equation}
By \cite[page 65]{gabcke-thesis}, we have $\beta < 0.93$ (more precisely,
$\beta = \cos (\pi/8)$). 
The lemma now follows from the inequality 
$\beta (2\pi)^{1/4}< 1.48$, and using
\begin{equation}
\begin{split}
\left|\sum_{n=1}^{n_1} 
 \frac{\cos(t\log n-\theta(t))}{\sqrt{n}}\right| 
 &\le  \frac{1}{2}\left|\sum_{n=1}^{n_1} 
  \frac{e^{it\log n-i\theta(t)}}{\sqrt{n}}\right| + 
\frac{1}{2} \left|\sum_{n=1}^{n_1} \frac{e^{i\theta(t)-it\log n}}{\sqrt{n}}\right|\\
&= \left|\sum_{n=1}^{n_1} n^{-1/2+it}\right|. 
\end{split}
\end{equation}
\end{proof}

In choosing the cutoff point 
for the intermediate range, we computed an approximation, 
call it $t_1$, for the largest $t \ge 200$ such that
the Riemann--Siegel--Lehman bound \eqref{rs bound}  
beats the van der Corput bound \eqref{vdc bound}.
This gave $t_1\approx t_0$.
To decide the leading constant in Theorem~\ref{zeta bound}, we 
computed an approximation, call it $c_1$, for the smallest
$c > 0$ such that the
Riemann--Siegel--Lehman bound evaluated 
at $t_1$ is smaller than $c t_1^{1/6}\log t_1$. This
gave $c_1\approx c_0$. 

The computational bound in Lemma~\ref{computational bound lemma} was verified
using interval arithmetic, which is also the method used 
in \cite{platt-trudgian}.

\begin{lemma}[Computational bound] \label{computational bound lemma}
If $0\le t\le 3$, then $|\zeta(1/2+it)|\le 1.461$. If $3\le t \le 200$, then
$|\zeta(1/2+it)| \le c_0\, t^{1/6}\log t$. 
\end{lemma}
\begin{proof}
We implemented 
the Euler-Maclaurin formula (see \cite{odlyzko-schonhage-algorithm,rubinstein-computational-methods})
using interval arithmetic. 
We remark that a substantial loss is incurred in 
the upper bound (and computational speed) due to the 
use of interval arithmetic, but this loss is still tolerable for our purposes.

We used a main sum of length $\lceil 2 t\rceil$ terms 
in the Euler-Maclaurin formula with a single correction term. 
Given an interval $[a_0,b_0]$,
our program 
computed an enveloping interval 
$I_{out}=[a_1,b_1]$ such that 
$\{|\zeta(1/2+it)|: t\in [a,b]\} \subset I_{out}$. Hence, 
\begin{equation}\label{interval bound}
\max_{t\in [a_0,b_0]}|\zeta(1/2+it)| \le \max_{t\in I_{out}} |t|=b_1.
\end{equation}
With this in mind,
we partitioned the interval $[3,200]$ into consecutive subintervals 
$I_q:=[q/Q,(q+1)/Q]$, where $q=3Q,\ldots,200Q-1$ and $Q=2^7$.
For each $I_q$, our program returned an enveloping interval $I_{q,out}$ 
which we used in 
 \eqref{interval bound} to verify that
\begin{equation}\label{check est}
\frac{\max_{t\in I_q}|\zeta(1/2+it)|}{(q/Q)^{1/6}\log(q/Q)}\le c_0.
\end{equation}
The estimate \eqref{check est} held for all relevant $q$, in fact, with a
smaller constant of $0.595$.

To verify the bound $|\zeta(1/2+it)|\le 1.461$ 
for $t \in [0,3]$, one could prove that $|\zeta(1/2+it)|$, $0\le t\le 3$,
attains its maximum at $t=0$. 
However, this seemed unduly complicated. Instead,
we fell back on
our interval arithmetic program after modifying some its 
parameters. Specifically, we used the Euler-Maclaurin formula 
with a main sum of $\lceil 60(t+1)\rceil$ terms, which is much longer than
before. 
We kept a single correction term, and 
also used a finer partition with $Q=2^{14}$.
The longer main sum 
for  $0\le t\le 3$ ensured that the error in the the Euler-Maclaurin 
formula was sufficiently small. Given this, we 
were able to verify the claimed bound.
\end{proof}

\begin{lemma}[Riemann--Siegel--Lehman bound] \label{trivial bound lemma}
If $t \ge 200$, then
\begin{equation}\label{rs bound}
|\zeta(1/2+it)| \le \frac{4t^{1/4}}{(2\pi)^{1/4}} - 2.08.
\end{equation}
In particular, for $200\le t\le t_0$, we have
$|\zeta(1/2+it)| \le c_0\, t^{1/6}\log t$.
\end{lemma}
\begin{proof}
This lemma gives a small improvement on the Lehman bound in \cite[Lemma
2]{lehman}. The improvement is in the term $-2.08$, which is significant if $t$
is not too large.

We start with the bound furnished by the Riemann--Siegel Lemma~\ref{riemann siegel lemma}. 
To bound the main sum there, we employ the estimate
\begin{equation}\label{integration step}
\begin{split}
2\Big|\sum_{n=1}^{n_1} \frac{e^{it\log n}}{\sqrt{n}}\Big| &\le 
2\sum_{n=1}^5 \frac{1}{\sqrt{n}}+ 2\int_5^{n_1}
\frac{1}{\sqrt{x}}\,dx\\
&= 2\sum_{n=1}^5 \frac{1}{\sqrt{n}}+ 4\sqrt{n_1}-4\sqrt{5}.
\end{split}
\end{equation}
Note that $n_1=\lfloor \sqrt{t/(2\pi)}\rfloor \ge 5$ for $t\ge 200$, so 
the integration step in \eqref{integration step} makes sense.
Also, we have
$2\sum_{n=1}^5 n^{-1/2} - 4\sqrt{5} < -2.48$.
To bound the remainder term $\mathcal{R}(t)$ in the Riemann--Siegel Lemma~\ref{riemann siegel
lemma}, we employ the estimate 
\begin{equation}\label{rem step}
\mathcal{R}(t)\le 1.48(200)^{-1/4} + 0.127(200)^{-3/4} < 0.4,\qquad (t\ge 200).
\end{equation}
The first part of the lemma now follows on 
substituting \eqref{integration step} and \eqref{rem step} back 
into \eqref{rs lemma}, and noting that $\sqrt{n_1}\le
\big(t/2\pi\big)^{1/4}$. 

To prove the second part of the lemma, we use \verb!Mathematica! to 
verify that
there is no solution to the equation
\begin{equation}
\frac{4t^{1/4}}{(2\pi)^{1/4}} - 2.08 = c_0\, t^{1/6}\log t,\qquad 
(200\le t\le t_0),
\end{equation}
and that the l.h.s of the equation is smaller than
the r.h.s. at $t=200$,
and therefore throughout the range $t\in [200,t_0]$.
\end{proof}

\begin{lemma}[Van der Corput bound] \label{corput bound lemma}
If $t \ge t_0$, then 
\begin{equation}\label{vdc bound}
|\zeta(1/2+it)| \le 
a_1 t^{1/6}\log t + a_2 t^{1/6}+ a_3
\end{equation}
where
$a_1:=0.6058490462530$, 
$a_2:=0.5743984045897$, 
and $a_3:= -2.884626766806$.
In particular, for $t\ge t_0$, we have $|\zeta(1/2+it)|\le c_0 t^{1/6}\log t$.
\end{lemma}
\begin{proof}
We plan to divide the main sum in the Riemann--Siegel
Lemma~\ref{riemann siegel lemma} into pieces of length $\approx t^{1/3}$,
then apply the van der Corput 
Lemma~\ref{corput lemma} to each piece. 
To this end, let 
$K =\lceil t^{1/3}\rceil$
and $R = \lfloor n_1/K\rfloor$, 
where, as before, $n_1=\lfloor \sqrt{t/2\pi}\rfloor$.
Here, $K$ is the length of each piece (except possibly the last one, which 
can be shorter), and $R+1$ is
the total number of pieces.

The remainder term 
$\mathcal{R}(t):=1.48 t^{-1/4}  + 0.127 t^{-3/4}$ 
in the Riemann--Siegel Lemma~\ref{riemann siegel lemma} 
satisfies
$\mathcal{R}\le \mathcal{R}(t_0)$ for $t\ge t_0$.
Thus, carrying out the aforementioned subdivision, 
and using the triangle inequality, we obtain
\begin{equation}\label{main sum division}
\begin{split}
|\zeta(1/2+it)| \le&  2\sum_{n=1}^{r_0 K-1} \frac{1}{\sqrt{n}}+2\sum_{r=r_0}^{R-1}
\Big|\sum_{n=rK}^{(r+1)K-1} \frac{e^{it\log n}}{\sqrt{n}} \Big|\\ 
&+2\Big|\sum_{RK}^{n_1} \frac{e^{it\log n}}{\sqrt{n}}\Big|+\mathcal{R}(t_0).
\end{split}
\end{equation}
Here, we trivially estimated the part of the 
main sum with $n<r_0K$,
where $r_0$ is a positive integer to be chosen later.
Also, we used $R>0$, which is due to $t\ge t_0$.

To bound the first sum in \eqref{main sum division},
we note that 
$r_0K\ge \lceil r_0 t_0^{1/3}\rceil$.
So, proceeding as in \eqref{integration step}, we obtain 
\begin{equation}
 2\sum_{n=1}^{r_0 K-1}  \frac{1}{\sqrt{n}} 
 \le 
 4\sqrt{r_0K}+ \mathcal{I}(r_0,t_0),
\end{equation}
where
\begin{equation}
\mathcal{I}(r_0,t_0) :=
2\sum_{n=1}^{\lceil r_0 t_0^{1/3}\rceil-1} \frac{1}{\sqrt{n}}
 -4\sqrt{\lceil r_0 t_0^{1/3}\rceil-1}.
\end{equation}
To bound the remaining sums in \eqref{main sum division}, 
we use partial summation~\cite[(5.2.1)]{titchmarsh}. 
Put together, if we let
\begin{equation}
S:=2\sum_{r=r_0}^R \frac{1}{\sqrt{rK}} \max_{\Delta \le K}
\Big|\sum_{k =0}^{\Delta-1} e^{it\log (rK+k)}\Big|, 
\end{equation}
then we obtain
\begin{equation}\label{zeta bound 0}
|\zeta(1/2+it)| \le S+ 4\sqrt{r_0K} +
\mathcal{I}(r_0,t_0)+\mathcal{R}(t_0).
\end{equation}

In order to bound the inner sum of $S$, 
we employ the van der Corput Lemma~\ref{corput lemma}.
Setting
\begin{equation}
f(x) := \frac{t}{2\pi}\log(rK+x), \qquad (0\le x \le \Delta-1),
\end{equation}
then
\begin{equation}
f'''(x) = \frac{t}{\pi (rK + x)^3}, \qquad (0\le x\le \Delta -1).
\end{equation}
So, on defining
\begin{equation}\label{W lambda def}
W :=W_r=\frac{\pi (r+1)^3K^3}{t},\qquad \lambda:=\lambda_r =\frac{(r+1)^3}{r^3},
\end{equation}
and noting that $\Delta \le K$, we obtain
\begin{equation}
\frac{1}{W} \le |f'''(y)|\le \frac{\lambda}{W},\qquad (0\le y\le K).
\end{equation}
We apply Lemma~\ref{corput lemma} 
with $L=K$, $\alpha_r:=\alpha(W_r,\lambda_r,\eta)$,
$\beta_r:=\beta(W_r,\eta)$, and with $\eta >0$ to be chosen later. 
 This yields
\begin{equation}
S \le  2\sum_{r=r_0}^{R} \frac{1}{\sqrt{rK}}
\sqrt{\alpha_r K^2 W_r^{-1/3} + \eta \alpha_r K + \beta_r K W_r^{1/3}
+\eta\beta_rW_r^{2/3}}.
\end{equation}
We factor out $K^2W_r^{-1/3} = \frac{K t^{1/3}}{\pi^{1/3}(r+1)}$ from under the
square-root. This gives
\begin{equation}\label{S bound CalB}
S\le  \frac{2 t^{1/6}}{\pi^{1/6}}\sum_{r=r_0}^{R} \sqrt{\frac{\mathcal{B}_r}{r(r+1)}},
\end{equation}
where
\begin{equation}
\mathcal{B}_r :=\alpha_r+\frac{\eta \alpha_r W_r^{1/3}}{K} + \frac{\beta_r W_r^{2/3}}{K}+\frac{\eta \beta_r W_r}{K^2}.
\end{equation}
Since $K< t^{1/3}+1$ and $R \le t^{1/6}/\sqrt{2\pi}$, 
we have
\begin{equation}
W_r \le \pi (R+1)^3 (1+t^{-1/3})^3\le \frac{\sqrt{t}\rho^3}{2^{3/2}\sqrt{\pi}},\qquad  
\rho := \Big(1+\frac{1}{R}\Big)\Big(1+\frac{1}{t^{1/3}}\Big).
\end{equation}
In addition, $K\ge t^{1/3}$, so put together we obtain
\begin{equation}\label{CalB}
\mathcal{B}_r \le \alpha_r+\frac{\eta\alpha_r \rho}{\sqrt{2}\pi^{1/6} t^{1/6}} +
\frac{\beta_r\rho^2}{2\pi^{1/3}} + \frac{\eta \beta_r\rho^3}{2^{3/2}\sqrt{\pi}t^{1/6}}.
\end{equation}
At this point, we make several observations.
First, $W_r$ is monotonically increasing in $r$, and $\lambda_r$
is monotonically decreasing in $r$.
Hence, by the definitions of $\alpha_r$ and $\beta_r$, we see that
they monotonically decrease with $r$.
So $\alpha_r \le \alpha_{r_0}$ and 
$\beta_r\le \beta_{r_0}$ for $r\ge r_0$. 
Second, we fix $\eta = \eta_0:= (75/64)^{2/3}$, which is to help balance the leading two terms in the formula for $\alpha_r$ in
\eqref{alpha beta formula}.
Third, if we define 
\begin{equation}
R_0 := \left\lceil \frac{\sqrt{t_0/(2\pi)}-1}{t_0^{1/3}+1}-1\right\rceil, 
\end{equation}
then $R\ge R_0$, and therefore 
\begin{equation}\label{rho max}
\rho \le \rho_0 := \Big(1+\frac{1}{R_0}\Big)\Big(1+\frac{1}{t_0^{1/3}}\Big). 
\end{equation}
Assembling these estimates into \eqref{S bound CalB}, and using
\begin{equation}
\sum_{r=r_0}^R\frac{1}{\sqrt{r(r+1)}}\le \log \frac{R}{r_0-1},
\end{equation}
together with the inequality $t\ge t_0$ to bound the denominators
in \eqref{CalB} from below, 
we obtain 
\begin{equation}\label{S estimates fin}
S \le \frac{2t^{1/6}}{\pi^{1/6}}\sqrt{\alpha_{r_0}
+ \frac{\eta_0 \alpha_{r_0} \rho_0}{\sqrt{2}\pi^{1/6} t_0^{1/6}}
+\frac{\beta_{r_0}\rho_0^2}{2\pi^{1/3}}
+  \frac{\eta_0 \beta_{r_0} \rho_0^3}{2^{3/2} \sqrt{\pi}
t_0^{1/6}}} \log \frac{R}{r_0-1}.
\end{equation}
We substitute \eqref{S estimates fin} back 
into \eqref{zeta
bound 0}, use the inequality $W_{r_0}\ge \pi(r_0+1)^3$
to simplify the bounds for $\alpha_{r_0}$ and $\beta_{r_0}$,
and choose $r_0=5$ which is suggested by 
numerical experimentation.
Then, we combine the resulting expression with the bounds 
\begin{equation}
4\sqrt{r_0K}\le 4\sqrt{r_0(1+t_0^{-1/3})}t^{1/6},\qquad
R \le t^{1/6}/\sqrt{2\pi},
\end{equation}
and numerically evaluate the resulting constants
using \verb!Mathematica!. 
On completion, this yields the first part of the lemma.

To prove the second part of the lemma, denote the r.h.s.\ of \eqref{vdc bound} by
$(*)$, and consider the equation $(*) = c_0\, t^{1/6}\log t$. 
Using \verb!Mathematica!, 
we find that there is no solution $t\ge
t_0$ to this equation, 
and that $(*)$ is smaller than the $c_0\, t^{1/6}\log t$ at $t=t_0$, 
and therefore for all $t\ge t_0$.
\end{proof}

\begin{remark}
The reason we restrict $t \ge t_0$ in Lemma~\ref{corput bound lemma}
is that the Riemann--Siegel--Lehman  bound from Lemma~\ref{trivial bound lemma}
is tighter for $t < t_0$. 
So we may as well take $t\ge t_0$, which 
 gives marginally better constants.
\end{remark}

\section{Proof of Lemma~\ref{corput lemma}}\label{lemma proof}

\begin{proof}
We use the Weyl-van der Corput 
Lemma in \cite[Lemma 5]{cheng-graham}, but in the more precise form presented at the bottom of
page 1273. 
Also, we incorporate a refinement pointed out by Platt and Trudgian 
in \cite[Lemma 2]{platt-trudgian}
that allows writing the leading term in \eqref{A process} 
as $L+M-1$ instead of $L+M$.
Put together, if $M$ is a positive integer, then
\begin{equation}\label{A process}
\begin{split}
\Big|\sum_{n=N+1}^{N+L} e^{2\pi i f(n)} \Big|^2 \le &  (L+M-1)\Big( 
\frac{L}{M}+\frac{2}{M}\sum_{m=1}^M \left(1-\frac{m}{M}\right)|S_m'(L)|\Big),
\end{split}
\end{equation}
where
\begin{equation}
S_m'(L) := \sum_{r = N+1}^{N+L-m} e^{2\pi i (f(r+m) - f(r))}.
\end{equation}
Henceforth, we may assume that $m< L$
and $L> 1$. Otherwise, the sum
$S_m'$ is empty, 
and so it does not contribute to the upper bound in \eqref{A process}.

Let $g(x) := f(x+m) - f(x)$,  where $N+1\le x\le  N+L-m$.
Then, $g''(x) = f''(x+m)-f''(x)$. 
Therefore, by the mean-value theorem, 
  $g''(x) = m f'''(\xi^*)$
for some $\xi^* \in (x, x + m)$.
Since we assumed that $1\le m<L$ and $L>1$ then
$(x,x+m)\subset [N+1,N+L]$.
So, given our bound on $f'''$, we deduce that
\begin{equation}
\frac{m}{W} \le |g''(x)| \le \frac{m\lambda}{W},\qquad 
(N+1\le x\le N+L-m).
\end{equation}
Applying the van der Corput Lemma in  
\cite[Lemma 3]{cheng-graham} to $S_m'(L)$
 thus yields  
\begin{equation}\label{corput second der}
|S_m'(L)| \le 
\frac{8\lambda L\sqrt{m/W}}{5}
+ \frac{3 \lambda L m }{W}+ \frac{8\sqrt{W/m}}{5}
+3.
\end{equation}
We substitute \eqref{corput second der} into \eqref{A process}, then execute
the summation over $m$ using the estimates  
that appear after \cite[Lemma 7]{cheng-graham}; namely, 
\begin{equation}\label{trapz rule}
\sum_{m=1}^M \left(1-\frac{m}{M}\right)\sqrt{m} \le \frac{4M^{3/2}}{15},
\qquad
\sum_{m=1}^M \left(1-\frac{m}{M}\right)\frac{1}{\sqrt{m}} \le
\frac{4\sqrt{M}}{3},
\end{equation}
as well as the Euler-Maclaurin summation estimates 
\begin{equation}
\sum_{m=1}^M \left(1-\frac{m}{M}\right)
<\frac{M}{2},\qquad  
\sum_{m=1}^M \left(1-\frac{m}{M}\right)m < \frac{M^2}{6}.
\end{equation}
From this, we conclude that
\begin{equation}\label{corput S' estimate}
\sum_{m=1}^M \left(1-\frac{m}{M}\right)|S_m'(L)| \le 
\frac{32\lambda L M^{3/2}}{75\sqrt{W}}
+ \frac{\lambda L M^2}{2W} 
+ \frac{32\sqrt{WM}}{15}  
+ \frac{3M}{2}.
\end{equation}
In particular, 
 substituting \eqref{corput S' estimate} back into \eqref{A process},
we obtain
\begin{equation*}
\Big|\sum_{n=N+1}^{N+L} e^{2\pi i f(n)} \Big|^2 \le  (L+M-1)\left( 
\frac{L}{M}+\frac{64\lambda L}{75}\sqrt{\frac{M}{W}}
+ \frac{\lambda L M}{W} 
+ \frac{64}{15}\sqrt{\frac{W}{M}}  
+ 3
\right).
\end{equation*}
We choose $M=\lceil \eta W^{1/3}\rceil$ 
for some free parameter $\eta >0$ that can be optimized 
(usually, $\eta$ will be around $1$).
This choice is in order to balance the first two 
terms on the r.h.s.\ as they typically dominate
in our application.
So, now, appealing to the inequality $\eta W^{1/3} \le
M\le \eta W^{1/3} + 1$, we deduce that
\begin{equation}
\begin{split}
\Big|\sum_{n=N+1}^{N+L} e^{2\pi i f(n)} \Big|^2 \le&
(L+\eta W^{1/3})\Big(\frac{L}{\eta W^{1/3}}+
\frac{64\lambda L}{75}\sqrt{\frac{\eta W^{1/3}+1}{W}}\\
&+\frac{\lambda L (\eta W^{1/3}+1)}{W}+
\frac{64}{15}\sqrt{\frac{W}{\eta W^{1/3}}}+
3\Big).
\end{split}
\end{equation}
We factor out $LW^{-1/3}$ from the first three terms in the second bracket,
and $W^{1/3}$ from the last two terms. This gives
\begin{equation*}\label{fin estimate}
\begin{split}
\Big|\sum_{n=N+1}^{N+L} e^{2\pi i f(n)} \Big|^2 
\le& (L+ \eta W^{1/3}) 
\frac{L}{W^{1/3}}\Big(\frac{1}{\eta}+\frac{64\lambda}{75} \sqrt{\eta +W^{-1/3}}\\
&+\frac{\lambda (\eta+W^{-1/3})}{W^{1/3}}\Big)
+  (L+ \eta W^{1/3})W^{1/3}\Big(\frac{64}{15\sqrt{\eta}}
+3W^{-1/3}\Big).
\end{split}
\end{equation*}
Last, the lemma follows on recalling the definitions of $\alpha$ and $\beta$ in
\eqref{alpha beta formula}.
\end{proof}

\bibliographystyle{amsplain}
\bibliography{explicitCorput}
\end{document}